\documentclass[a4paper,12pt,reqno]{amsart} 

\usepackage{amssymb,latexsym}
\usepackage{amsmath,amsthm}
\usepackage{enumerate}
\usepackage[mathscr]{eucal}
\usepackage{subcaption}
\usepackage{wrapfig}
\usepackage{graphicx}
\usepackage{multicol}
\usepackage{gensymb}
\usepackage{url}

\usepackage{graphics} 

\theoremstyle{definition}
\newtheorem{theorem}{\indent Theorem}[section]

\newtheorem{lemma}[theorem]{\indent Lemma}
\theoremstyle{definition}
\newtheorem{definition}[theorem]{\indent Definition}
\newtheorem{example}[theorem]{\indent Example}
\theoremstyle{definition}
\newtheorem{remark}[theorem]{\indent Remark}

\allowdisplaybreaks

\begin{document}
\null
\vskip 1.8truecm

\title[Some remarks on anti-topological spaces]
{Some remarks \\ on anti-topological spaces} 
\author{Tomasz Witczak}    
\address{Institute of Mathematics\\ Faculty of Science and Technology
\\ University of Silesia\\ Bankowa~14\\ 40-007 Katowice\\ Poland}
\email{tm.witczak@gmail.com} 


\begin{abstract}This paper is devoted to a general presentation of anti-topological spaces. These structures have been initially proposed by \c{S}ahin, Karg{\i}n and M. Y\"{u}cel in 2021. We analyse their basic definition, showing some of its subtleties and implications. The framework thus obtained is used to investigate anti-topological interpretation of some basic topological notions. For example, we discuss the idea of interior and closure and we show some results on door spaces. Moreover, we introduce two non-equivalent types of continuity. We investigate the idea of density and nowhere density. Finally, we give some preliminary suggestions concerning the modal logic of anti-topological spaces. 

It is noteworthy that the paper contains some additional remarks on infra-topological and weak spaces. They may be considered as a clarification or correction of some earlier results present in literature. 

\end{abstract}



\maketitle


\section{Introduction}

Anti-topological spaces have been defined by \c{S}ahin, Karg{\i}n and M. Y\"{u}cel in \cite{SAHIN}. These structures have been introduced together with neutro-topological spaces. The authors studied some basic properties of these two classes and the most striking relationships between them. In this paper we shall concentrate only on anti-topologies (without any special references to the matter of fuzziness and similar concepts like intuitionistic fuzziness, softness or neutrosophy).  

Undoubtedly, the past three decades were rich in works addressing the idea of generalization of the initial notion of topological space. All these studies can be easily justified. First, they show us which conditions are really important for the preservation of some basic topological properties (and which are superflous). 

Second, they provoke some kind of discussion on the mathematical, logical and philosophical meaning of the terms used in topology (like openness, closeness, interior, closure, density, nowhere density, open cover etc.). This is because some of these objects, operations and properties behave in an untypical way when they are applied to various generalized structures. For example, the interior of a set need not to be open. The closure of empty set need not to be empty. The interior of the whole universe may not be identical with this universe etc.

Third, this line of research is helpful when it comes to classification of numerous types of sets and their families. This becomes especially evident when one considers that various weak forms of open sets (that is, $\alpha$-, pre-, $\beta$-, semi- or $b$- open sets) have been adapted in many generalized settings. 

Generalization of the notion of topology relies on the assumption that we can \emph{remove} some of the conditions which constitute the family of open sets. For example, we can give up the assumption of closure under arbitrary unions (to obtain \emph{infra}-topologies\footnote{Cs\'asz\'ar named them \emph{quasi-topologies}.}, see \cite{ODH}, \cite{CONT} and \cite{DHANA}) or finite intersections (to get \emph{supra}-topologies, see \cite{CSASZAR} and \cite{MASS}). We may drop both these restrictions and this step gives us \emph{minimal} structures (see \cite{POPA}). If the only requirement is openness of empty set, then we have \emph{weak structures} (see \cite{WEAK}). Finally, we may turn our attention to \emph{generalized weak structures} which are arbitary families of subsets. These spaces have been introduced ten years ago by Avila and Molina in \cite{AVILA}. Later, they have been studied in \cite{JAMUNA}. However, it seems that already in 1966 Kim-Leong Lim took on the task of reconstructing topology on the basis of the most general assumptions (see \cite{KIM}). He assumed that \emph{generalized topology} is just any subset of $P(X)$. This is identical with the approach of Avila and Molina. Hence, his contribution should be appreciated by modern authors.

However, generalization is not the only possible modification. In this paper we want to think over the idea of anti-topological space. All the spaces mentioned above are connected with the concept of \emph{closure} under certain operations or with the assumption that some distinguished sets (like empty set or the whole universe) necessarily belong to our family. Anti-topological strategy reverses this approach (at least in some sense). What is constitutive for anti-topology, is the fact that intersections and unions of elements of the family in question, are \emph{beyond} (anti-)topology. Again, we can ask: what does it mean for the basic notions mentioned earlier (like openness or density)? In this paper we give several answers and additional suggestions. Our aim is to present some kind of general framework which could be later used both by us and other authors. In fact, we have already extended this results in \cite{WITCZAK} (where we analysed the notion of rarity in anti-topological spaces).

\section{Basic notions}

In general, the very basic definition of anti-topological space is taken from \cite{SAHIN}. However, we would like to discuss some issues which were not covered by the authors. Moreover, we would like to expand their initial research in a significant way. 

\begin{definition}
\label{antidef}
Let $X$ be a non-empty universe and $\mathcal{T}$ be a collection of subsets of $X$. We say that $(X, \mathcal{T})$ is an anti-topological space if the following conditions are satisfied:

\begin{enumerate}
\item $\emptyset, X \notin \mathcal{T}$.

\item For any $n \in \mathbb{N}$, if $A_1, A_2, ..., A_n \in \mathcal{T}$, then $\bigcap_{i=1}^{n} A_i \notin \mathcal{T}$ (with the assumption that the sets in question are not all identical, i.e. the intersection is non-trivial)\footnote{This assumption was not mentioned by the authors in their original paper. However, its necessity is clear.}.

\item For any collection $\{A_i\}_{i \in J \neq \emptyset}$ such that $A_i \in \mathcal{T}$ for each $i \in J$, $\bigcup_{i \in J} A_i \notin \mathcal{T}$ (with the assumption that the sets in question are not all identical, i.e. the union is non-trivial). 

\end{enumerate}
\end{definition}

We call the elements of $\mathcal{T}$ \emph{anti-open} sets, while their complements are \emph{anti-closed} sets. The set of all anti-closed sets (with respect to a given anti-topology) will be denoted by $\mathcal{T}_{Cl}$. We say that every anti-topology is \emph{anti-closed} under finite intersections and arbitrary unions (this refers respectively to Cond. (2) and Cond. (3) from the definition above). Attention: we assume that the property of being anti-closed refers only to non-trivial intersections or unions. We will use the notion of \emph{non-trivial family} to speak about those families of sets which contain at least two (different) sets.

Each anti-topology is connected with some \emph{associated space} denoted by $\tau$: the one which contains the empty set, the whole universe and all the finite intersections and arbitrary unions of the sets from $\mathcal{T}$. Clearly, this space is a topological one. Moreover, $\mathcal{T}$ is always contained in $\tau$. 

\begin{example}
\label{exlist}
Let us list down some examples of anti-topological spaces. The first one is taken from \cite{SAHIN}, the rest is our own invention. 

\begin{enumerate}
\item Let $X = \{1, 2, 3, 4\}$ and $\mathcal{T} = \{ \{1, 2\}, \{2, 3\}, \{3, 4\} \}$. Clearly, the only possible intersections are $\{1, 2\} \cap \{2, 3\} = \{2\} \notin \mathcal{T}$, $\{1, 2\} \cap \{3, 4\} = \emptyset \notin \mathcal{T}$ and $\{2, 3\} \cap \{3, 4\} = \{3\} \notin \mathcal{T}$. As for the unions, these are $\{1, 2\} \cup \{2, 3\} = \{1, 2, 3\} \notin \mathcal{T}$, $\{1, 2\} \cup \{3, 4\} = \{1, 2, 3, 4\} = X \notin \mathcal{T}$ and $\{2, 3\} \cup \{3, 4\} = \{2, 3, 4\} \notin \mathcal{T}$. 

We have $\mathcal{T}_{Cl} = \{ \{1, 2\}, \{1, 4\}, \{3, 4\} \}$. Note that $\{1, 2\}$ and $\{3, 4\}$ are both anti-open and anti-closed. 

As for the associated space, it is 

$\tau_{\mathcal{T}} = \{ \emptyset, X, \{1, 2\}, \{2, 3\}, \{3, 4\}, \{2\}, \{3\}, \{1, 2, 3\}, \{2, 3, 4\} \}$. 

\item Let $X = \{1, 2\}$ and $\mathcal{T} = \{ \{1\}, \{2\} \} = \mathcal{T}_{Cl}$. Clearly, $\{1\} \cap \{2\} = \emptyset \notin \mathcal{T}$ and $\{1\} \cup \{2\} = X \notin \mathcal{T}$. 

\item Let $X = \{a, b, c, d, e\}$ and $\mathcal{T} = \{ \{a, b\}, \{c, d\}, \{e\} \}$. Then $\mathcal{T}$ is an anti-topology on $X$. Note that each intersection of non-identical elements of this family is empty. As for the $\mathcal{T}_{Cl}$, it is $\{ \{c, d, e\}, \{a, b, e\}, \{a, b, c, d\} \}$. 

The associated space is $\tau_{\mathcal{T}} = \{ \emptyset, X, \{a, b\}, \{c, d\}, \{e\}, \{a, b, c, d\}, \\ \{a, b, e\}, \{c, d, e\} \}$. 

\item Let $X = \mathbb{N}^{+}$ and assume that $\mathcal{T}_{k}$ consists only of these finite subsets of $X$ which have cardinality $k$, where $k$ is a fixed positive natural number. Now, if $A, B \in \mathcal{T}_{k}$ and $A \neq B$, then their union has cardinality $m > k$ and their intersection has cardinality $n < k$. Clearly, $\emptyset \notin \mathcal{T}_{k}$ and $X = \mathbb{N}^{+} \notin \mathcal{T}_{k}$. Of course we may replace $\mathbb{N}^{+}$ with $\mathbb{N}$. 

\item Let $X = \mathbb{N}^{+}$ and $\mathcal{T} = \{ \{1\}, \{2\}, \{3\}, \{4\}, ... \} = \{ \{n\}, n \in \mathbb{N}^{+} \}$. This is a special case of $\mathcal{T}_{k}$ defined above (for $k = 1$). 

\item Let $X$ be arbitrary and $\mathcal{T} = \{ \{x\}, x \in X \}$. This is just a collection of all singletons of the elements of an arbitrary universe. 

\item Let $X = \mathbb{R}$ and assume that $\mathcal{T}_{y}$ consists only of these closed intervals which have length $y$, where $y$ is a fixed positive real number. Now, if $A, B \in \mathcal{T}_{y}$ and $A \neq B$, then their union has length $z > y$ (moreover, it is possible that it is not an interval at all) and their intersection has length $w < y$ (moreover, it can be empty or consist of one point). Clearly, $\emptyset$ and $X = \mathbb{R}$ do not belong to $\mathcal{T}_{y}$. 

\item Let $X = \mathbb{R}^2$ with usual Euclidean metric. Assume that $\mathcal{T}_{r}$ consists of all these closed balls which have radius $r$, where $r$ is some fixed positive real number. Any union of such balls has radius bigger than $r$ (or is not a ball at all). Any intersection has radius smaller than $r$ (or is not a ball at all). 

\item Let $X = \mathbb{N}^{+}$ and $\mathcal{T} = \{ \{1, 2\}, \{2, 3\}, \{3, 4\}, \{4, 5\}, ... \} = \{ \{n, n+1\}, n \in \mathbb{N}^{+} \}$. 

\item Let $X = \mathbb{R}$ and $\mathcal{T} = \{ \mathbb{R}^{-}, \mathbb{R}^{+} \}$. Clearly, $\mathbb{R}^{-} \cap \mathbb{R}^{+} = \emptyset \notin \mathcal{T}$ and $\mathbb{R}^{-} \cup \mathbb{R}^{+} = \mathbb{R} \setminus \{0\} \notin \mathcal{T}$. 

\item Let $X$ be arbitrary and $\emptyset \neq A \subseteq X$. Then we may define anti-topology $\mathcal{T} = \{A, X \setminus A\}$. 

\end{enumerate}
\end{example}

We would like to point out some properties which are simple but maybe not visible at first glance.

\begin{lemma}
Let $(X, \mathcal{T})$ be an anti-topological space. Then $|X| > 1$. 
\end{lemma}

\begin{proof}
Clearly, if $|X| = 1$, then $X$ is of the form $\{a\}$ for some $a$. The only possible subsets are $\emptyset$ and $X$ itself which are not allowed because of the very definition of anti-topology.
\end{proof}

\begin{lemma}
Assume that $(X, \mathcal{T})$ is an anti-topological space, $B \in \mathcal{T}$ and $A \subseteq B$. Then $A \notin \mathcal{T}$. 
\end{lemma}

\begin{proof}
If $A \subseteq B$ then $A = A \cap B$ and $A \cap B \notin \mathcal{T}$.
\end{proof}

One may compare this result with the idea of \emph{minimal} open sets (see \cite{CARPIN}).

\begin{lemma}
Assume that $X$ is a non-empty universe and $\mathcal{U}$ is a family of subsets of $X$ that is anti-closed under finite intersections. Then it is anti-closed under arbitrary intersections. 
\end{lemma}

\begin{proof}
Assume that there exists certain family $\{ A_i \}_{i \in J \neq \emptyset}$ such that $|J| \geq \aleph_{0}$, for any $i \in J$, $A_i \in \mathcal{U}$ and $A = \bigcap_{i \in J} A_i \in \mathcal{U}$. Then take $A_k \neq A$ for some $k \in J$. Now we may write that $A_k \cap A = A$. This is binary intersection, hence $A \notin \mathcal{U}$. Contradiction.
\end{proof}

Clearly, the lemma above applies to anti-topologies too. Moreover, we can show that Cond. (2) and (3) from Def. \ref{antidef} are equivalent. 

\begin{lemma}
Let $X$ be a non-empty universe. Let $\mathcal{U}$ be a family of subsets of $X$ which is anti-closed under finite intersections. Then it is anti-closed under arbitrary unions. 
\end{lemma}

\begin{proof}
Assume that there exists some non-trivial family $\{ A_i \}_{i \in J \neq \emptyset}$ such that for any $i \in J$, $A_i \in \mathcal{U}$ and $A = \bigcup_{i \in J} A_i \in \mathcal{U}$. Then take $A_k \neq A$ for some $k \in J$. Then $A_k \cap A \notin \mathcal{U}$. However, $A_k \cap A = A_k$ and we assumed that $A_k \in \mathcal{U}$ (just like any other element of $\{A_i \}$). This is contradiction.
\end{proof}

\begin{lemma}
Let $X$ be a non-empty universe. Let $\mathcal{U}$ be a family of subsets of $X$ which is anti-closed under arbitrary unions. Then it is anti-closed under finite intersections. 
\end{lemma}

\begin{proof}
Assume that there are two different subsets of $X$, namely $A$ and $B$, such that $A, B \in \mathcal{U}$ and $A \cap B \in \mathcal{U}$. Then consider $A \cup (A \cap B) = A$. By virtue of anti-closure under unions, $A \notin \mathcal{U}$. This is contradiction. Note that it was enough to assume anti-closure under finite unions. 
\end{proof}

We can check some properties of anti-closed sets. 

\begin{lemma}
Assume that $(X, \mathcal{T})$ is an anti-topological space and $A, B \in \mathcal{T}_{Cl}$. Suppose that $A \neq B$. Then $A \cap B \notin \mathcal{T}_{Cl}$. 
\end{lemma}

\begin{proof}
If $A, B \in \mathcal{T}_{Cl}$, then $-A, -B \in \mathcal{T}$. Assume that $A \cap B \in \mathcal{T}_{Cl}$. Then $-(A \cap B) \in \mathcal{T}$. But then $-A \cup -B \in \mathcal{T}$ and this is contradiction.
\end{proof}

\begin{lemma}
Assume that $(X, \mathcal{T})$ is an anti-topological space and $\{ A_i\}_{i \in J} \subseteq \mathcal{T}_{Cl}$. Then $\bigcup_{i \in J} A_i \notin \mathcal{T}_{Cl}$.
\end{lemma}

\begin{proof}
Assume that $\bigcup_{i \in J} A_i \in \mathcal{T}_{Cl}$. Then $-\bigcup_{i \in J} A_i \in \mathcal{T}$. Hence (by virtue of De Morgan's laws) $\bigcap_{i \in J} (-A_i) \in \mathcal{T}$. But for any $i \in J$, $-A_i \in \mathcal{T}$, hence their intersection should be beyond $\mathcal{T}$. This is contradiction. 
\end{proof}

\section{Anti-interior and anti-closure}

In this section we define anti-interior and anti-closure of a set in anti-topological space. 

\begin{definition}
Assume that $(X, \mathcal{T})$ is an anti-topological space and $A \subseteq X$. Then we define anti-interior of $A$ (that is, $aInt(A)$) and its anti-closure (namely, $aCl(A)$) as follows:

\begin{enumerate}
\item $aInt(A) = \bigcup \{U; U \subseteq A \text{ and } U \in \mathcal{T}\}$. 

\item $aCl(A) = \bigcup \{F; A \subseteq F \text{ and } F \in \mathcal{T}_{Cl} \}$. 

\end{enumerate}

\end{definition}

\begin{example}
\label{wnetrze}
Some examples of anti-interior and anti-closure are presented below:

\begin{enumerate}
\item Let $(X, \mathcal{T})$ be like in Example \ref{exlist} (1). Consider $A = \{1, 2, 3\}$. Then $aInt(A) = \{1, 2\} \cup \{2, 3\} = A \notin \mathcal{T}$. As we can see, anti-interior may not be anti-open. Now $aCl(A) = \bigcap \emptyset = X$. 

\item Let $(X, \mathcal{T})$ be like in Example \ref{exlist} (3). Consider $A = \{a, b, c\}$. Now $aInt(A) = \{a, b\}$ and $aCl(A) = \{a, b, c, d\}$. 

\item Let $(X, \mathcal{T})$ be like in Example \ref{exlist} (4). Note that for any $A \subseteq \mathbb{N}$ such that $|A| \geq k \in \mathbb{N}^{+}$, the following holds: $aInt(A) = A$. This is because $A$ is of the form $\{a_1, a_2, ..., a_n \}$, $n > k$. Hence, it can be presented as a union of all its subsets of cardinality $k$. For example, if $k = 2$ and $A = \{10, 12, 20\}$, then $aInt(A) = \{10, 12\} \cup \{10, 20\} \cup \{12, 20\} = A$. 
\end{enumerate}

\end{example}

We may easily predict some of the basic properties of anti-interior and anti-closure.

\begin{theorem}
\label{properties}
Let $(X, \mathcal{T})$ be an anti-topological space. Let $A \subseteq X$. Then the following statements are true:

\begin{enumerate}
\item $aInt(A) \subseteq A$. 
\item If $A \in \mathcal{T}$, then $aInt(A) = A$. 
\item If $A \subseteq B$, then $aInt(A) \subseteq aInt(B)$ (monotonicity). 
\item $aInt(aInt(A)) = aInt(A)$ (idempotence). 
\item $A \subseteq aCl(A)$. 
\item If $-A \in \mathcal{T}$, then $aCl(A) = A$. 
\item If $A \subseteq B$, then $aCl(A) \subseteq aCl(B)$. 
\item $aCl(aCl(A)) = aCl(A)$. 
\item $-aInt(A) = aCl(-A)$. 
\item $aInt(-A) = -aCl(A)$.
\item $x \in aInt(A)$ if and only if there is $U \in \mathcal{T}$ such that $x \in U \subseteq A$. 
\item $x \in aCl(A)$ if and only if $U \cap A \neq \emptyset$ for any $U \in \mathcal{T}$ such that $x \in U$. 
\end{enumerate}
\end{theorem} 

\begin{proof}
All these properties are true in any generalized weak structure (see \cite{AVILA} and \cite{JAMUNA}). It means that they are true for any $\mathfrak{g} \subseteq P(X)$. Hence, they are true in anti-topological framework too. The only important thing is to define interior and closure in a standard manner.
\end{proof}

\begin{lemma}
\label{inter}
Assume that $(X, \mathcal{T})$ is an anti-topological space. Then $aInt(A \cap B) \subseteq aInt(A) \cap aInt(B)$.
\end{lemma}

\begin{proof}
Let $A, B \subseteq X$. Now $A \cap B \subseteq A$, $A \cap B \subseteq B$. From the monotonicity of interior we get that $aInt(A \cap B) \subseteq aInt(A)$ and $aInt(A \cap B) \subseteq aInt(B)$. Hence $aInt(A \cap B) \subseteq aInt(A) \cap aInt(B)$. 
\end{proof}

Note that the lemma above is true in any generalized weak structure. The converse is not necessarily true. Consider $X = \{1, 2, 3, 4, 5\}$, $\mathcal{T} = \{ \{1, 3\}, \{2\}, \{3, 4\} \}$, $A = \{1, 2, 3\}$ and $B = \{2, 3, 4\}$. We have $aInt(A) = \{1, 3\} \cup \{2\} = \{1, 2, 3\}$ and $aInt(B) = \{2\} \cup \{3, 4\} = \{2, 3, 4\}$. Moreover, $A \cap B = \{2, 3\}$ and $aInt(A \cap B) = \{2\}$. Now $aInt(A) \cap aInt(B) = \{2, 3\} \nsubseteq \{2\}$. 

\begin{remark}
Note that in topological spaces the converse of the lemma analogous to Lemma \ref{inter} could be proved in the following manner. First, $Int(A) \cap Int(B) \subseteq A \cap B$. This is obvious. Second, $Int(Int(A) \cap Int(B)) \subseteq Int(A \cap B)$. However, both $Int(A)$ and $Int(B)$ are open, hence their intersection is open too. Thus $Int(Int(A) \cap Int(B)) = Int(A) \cap Int(B)$ and we are done. 

This proof is not true in these spaces where interior may not be open. However, if we assume that our generalized weak structure (say, $\mathfrak{g}$) is closed under finite intersections (like in the case of infra-topologies), then we may use the following reasoning. Let $x \in Int(A) \cap Int(B)$. So there are $C, D \in \mathfrak{g}$ such that $x \in C \cap D$, $C \subseteq A$ and $D \subseteq B$. Now $C \cap D$ is open and contained in $A \cap B$. Thus $x \in Int(A \cap B)$. 

\end{remark}

\begin{remark}
As for the Lemma \ref{inter}, it can be easily generalized to the following form (see \cite{WITCZAK}): assume that $(X, \mathcal{T})$ is an anti-topological space and $\{A_i\}_{i \in J \neq \emptyset}$ is a family of sets. Then $aInt(\bigcap_{i \in J} A_i) \subseteq \bigcap_{i \in J} aInt(A_i)$. 

An analogous lemma may be formulated for generalized weak structures too. 
\end{remark}

\begin{lemma}
Let $(X, \mathcal{T})$ be an anti-topological space and $A, B \subseteq X$. Then $aInt(A) \cup aInt(B) \subseteq aInt(A \cup B)$. 
\end{lemma}

\begin{proof}
Let $x \in aInt(A) \cup aInt(B)$. Without loss of generality we may assume that there is some $C \in \mathcal{T}$ such that $C \subseteq A$ and $x \in C$. Then $C \subseteq A \cup B$. Moreover, $C$ is anti-open, hence $x \in aInt(A \cup B)$.
\end{proof}

As for the converse of Theorem \ref{properties} (2), it may be false (as we could see in Example \ref{wnetrze}). Thus we may introduce the following definition (\emph{per analogiam} with infra-topological structures, see \cite{WIT}).

\begin{definition}
Let $(X, \mu)$ be an anti-topological space. Let $A \subseteq X$. If $aInt(A) = A$ then we say that $A$ is \emph{pseudo-anti-open}. If $aInt(A) \in \mathcal{T}$ then we say that $A$ is \emph{anti-genuine}. 
\end{definition}

\begin{lemma}
Let $(X, \mathcal{T})$ be an anti-topological space. Let $A \subseteq X$. If $aInt(A)$ may be written as a union of two or more anti-open sets then $aInt(A) \notin \mathcal{T}$. 
\end{lemma}

Note that in Example \ref{wnetrze} (1) the set $\{1, 2, 3\}$ is pseudo-anti-open and is not anti-genuine. In Example \ref{wnetrze} (2) the set $\{a, b, c\}$ is anti-genuine and is not pseudo-anti-open. 

Let us check some properties of pseudo-anti-open and anti-genuine sets.

\begin{lemma}
Let $(X, \mathcal{T})$ be an anti-topological space. Then every anti-open set is pseudo-anti-open and anti-genuine. 
\end{lemma}

\begin{lemma}
Assume that $(X, \mathcal{T})$ is an anti-topological space and $\{A_i\}_{i \in J \neq \emptyset}$ is a family of pseudo-anti-open sets. Then $\bigcup_{i \in J} A_i$ is pseudo-anti-open too.
\end{lemma}

\begin{proof}
Clearly, $aInt(\bigcup_{i \in J} A_i) \subseteq \bigcup_{i \in J} A_i$. Now assume that $x \in \bigcup_{i \in J} A_i$ but $x \notin aInt(\bigcup_{i \in J} A_i$. Hence there is some $k \in J$ such that $x \in A_k$ but for any anti-open $G \subseteq \bigcup_{i \in J} A_i$, $x \notin G$. However, $A_k = Int(A_k)$. Hence, there is $B \in \mathcal{T}$ such that $x \in B \subseteq A$. But then $B \subseteq \bigcup_{i \in J} A_i$. 
\end{proof}

\begin{lemma}
Assume that $(X, \mathcal{T})$ is an anti-topological space and $A, B \subseteq X$ are anti-genuine. Assume that $aInt(A \cap B)$ is different than $aInt(A)$ and $aInt(B)$. Then $A \cap B$ is not anti-genuine.
\end{lemma}

\begin{proof}
If $A$ and $B$ are anti-genuine, then $aInt(A) \in \mathcal{T}$ and $aInt(B) \in \mathcal{T}$. We already know that $aInt(A \cap B) \subseteq aInt(A) \cap aInt(B)$. Now assume that $aInt(A \cap B) \in \mathcal{T}$. But $aInt(A \cap B)$ may be written as $aInt(A \cap B) \cap aInt(A)$. This is an intersection of two different anti-open sets, hence it cannot belong to $\mathcal{T}$. 
\end{proof}

The assumption expressed in the second sentence of this lemma is important. Consider $\mathcal{T}$ from Example \ref{exlist} (3). Let $A = \{a, b, c\}$, $B = \{c, d\}$ and $C = \{a, b\}$. All these sets are anti-genuine. Now $aInt(A \cap B) = aInt(\{c\}) = \emptyset \notin \mathcal{T}$. Clearly, $aInt(A \cap B) \neq aInt(A)$ and $aInt(A \cap B) \neq aInt(B)$. On the contrary, $aInt(A \cap C) = aInt(C) = \{a, b\} \in \mathcal{T}$.  

It is possible that the union of two anti-genuine sets is not anti-genuine. Consider the same $\mathcal{T}$. We see that $aInt(A \cup B) = aInt(\{a, b, c, d\}) = \{a, b, c, d\} \notin \mathcal{T}$. But this is not general because $aInt(A \cup C) = aInt(\{a, b, c\}) = \{a, b\} \in \mathcal{T}$. \\

\quad \\

One can use anti-interior and anti-closure to define other classes of sets. In general, this is beyond the scope of this initial research but we may show some clues. 

\begin{definition}
Let $(X, \mathcal{T})$ be an anti-topological space and $A \subseteq X$. We say that $A$ is semi-open if and only if $A \subseteq aCl(aInt(A))$. 
\end{definition}

\begin{remark}
The idea of semi-open sets is taken from other families (like topologies or generalized topologies). Modak used it in the context of weak structures (see \cite{WEAK}). However, he wrote in Lemma 3.1. that if $A$ is semi-open in such structure, then $Int(A) \neq \emptyset$. But we may consider the following weak structure: $X = \{a, b, c\}$ and $\omega = \{ \emptyset, \{a\} \}$. Now consider $A = \{b, c\}$. Of course $Int(A) = \emptyset$. Then $Cl(\emptyset) = \{a, b, c\} \cap \{b, c\} = \{b, c\} = A$. Hence $Int(A) = \emptyset$ and $A \subseteq Cl(Int(A)) = A$. The closure of empty set need not to be empty in weak structure. 

We may adjust this example to anti-topologies. Take the same $X$ and $\mathcal{T} = \{ \{a\} \}$. Take the same $A$. Now $aInt(A) = \emptyset$ and $aCl(\emptyset) = \{b, c\} = A$. 
\end{remark}

\section{Door anti-topologies}

Door spaces (when defined in topological or supra-topological environment) are defined by the assumption that each subset is open or closed. This definition is not reasonable in anti-topological context because $\emptyset$ and $X$ are never open nor closed. However, we can make it more useful.

\begin{definition}
Let $(X, \mathcal{T})$ be an anti-topological space. We say that $(X, \mathcal{T})$ is \emph{door anti-topological space} if and only if each subset (different than $\emptyset$ and $X$) is anti-open or anti-closed. 
\end{definition}

\begin{example}
Here there are some examples of door spaces:

\begin{enumerate}

\item $X = \{a, b\}$, $\mathcal{T} = \{ \{a\}, \{b\} \}$. 

\item $X = \{a, b, c\}$, $\mathcal{T} = \{ \{a\}, \{b\}, \{c\} \}$. 

\item $X = \{a, b, c\}$, $\mathcal{T} = \{ \{a, b\}, \{a, c\}, \{b, c\} \}$. 

\end{enumerate}
\end{example}

We may prove the following theorem:

\begin{theorem}
Assume that $(X, \mathcal{T})$ is an anti-topological space such that $|X| > 3$. Then $(X, \mathcal{T})$ cannot be door space. 
\end{theorem}

\begin{proof}
Take arbitrary $x, y, z \in X$ and consider $\{x\}$. Assume that $\{x\}$ is anti-open. Now $\{x, y\}$ cannot be open (as a non-trivial superset of anti-open set). Moreover, it must be anti-closed (because of the door property). The same can be said about $\{x, y, z\}$ (note that this set is different than the whole universe because we assumed that $|X| > 3$). Now $\{x, y, z\} \cap \{x, y\} = \{x, y\}$. On the one hand, it is anti-closed. On the other, it cannot be anti-closed because any non-trivial intersection of two anti-closed sets is not anti-closed. Contradiction.

Now assume that $\{x\}$ is anti-closed. Then $\{x, y\}$ cannot be anti-closed. Hence, it must be anti-open. The same can be said about $\{x, y, z\}$. But then $\{x, y\} \cap \{x, y, z\} = \{x, y\}$ cannot be anti-open. Again, contradiction.
\end{proof}

\section{Continuity}

In this section we show some initial results on anti-continuity.

\begin{definition}
Let $(X, \mathcal{T})$ and $(Y, \mathcal{S})$ be two anti-topological spaces. We say that a function $f: X \to Y$ is \emph{anti-continuous} if and only if for any $A \in \mathcal{S}$, $f^{-1}(A) \in \mathcal{T}$.
\end{definition}

\begin{example}
Below there are some examples of anti-continuity.

\begin{enumerate}
\item Let $X = \{1, 2, 3, 4\}$, $\mathcal{T} = \{ \{1, 2\}, \{3, 4\}, \{5\} \}$, $Y = \{a, b, c, d, e\}$, $\mathcal{S} = \{ \{a, b\}, \{c, d\}, \{e\} \}$. Let $f$ be defined in the following manner: $f(1) = a$, $f(2) = b$, $f(3) = c$, $f(4) = d$ and $f(5) = e$. 

Now $f^{-1}(\{a, b\}) = \{1, 2\} \in \mathcal{T}$, $f^{-1}(\{c, d\}) = \{3, 4\} \in \mathcal{T}$ and $f^{-1}(\{e\}) = \{5\} \in \mathcal{T}$. 

\item Let $X = \{1, 2, 3, 4\}$, $\mathcal{T} = \{ \{1, 2\}, \{3\} \}$, $Y = \{a, b, c, d, e\}$, $\mathcal{S} = \{ \{a, b, c, d\}, \{e\} \}$. Let $f(1) = a$, $f(2) = b$ and $f(3) = e$. 

Now $f^{-1}(\{a, b, c, d\}) = \{1, 2\} \in \mathcal{T}$ and $f^{-1}(\{e\}) = \{3\} \in \mathcal{T}$. 

\item Let $X = \mathbb{N}$, $\mathcal{T} = \{ \{0, 1\}, \{1, 2\}, \{2, 3\}, ... \}$. Let $Y = \mathbb{R}^{+} \cup \{0\}$ and $\mathcal{S}$ be a collection of all closed intervals of length $1$ such that their endpoints are natural numbers. For example, $[0, 1]$, $[1, 2]$, $[2, 3]$ belong to $\mathcal{S}$. Let $f: X \to Y$ be defined as $f(n) = n$. Now $f^{-1}([0, 1]) = \{0, 1\} \in \mathcal{T}$, $f^{-1}([3,4]) = \{3, 4\} \in \mathcal{T}$ etc.

\end{enumerate}

\end{example}

Another type of continuity is this one.

\begin{definition}
Let $(X, \mathcal{T})$ and $(Y, \mathcal{S})$ be two anti-topological spaces. We say that a function $f: X \to Y$ is \emph{point-anti-continuous} if and only if for any $x \in X$ and for any $V \in \mathcal{S}$ such that $f(x) \in V$, there is $U \in \mathcal{T}$ such that $x \in U$ and $f(U) \subseteq V$. 

\end{definition}

\begin{theorem}
Let $(X, \mathcal{T})$ and $(Y, \mathcal{S})$ be two anti-topological spaces. Assume that $f: X \to Y$ is anti-continuous. Then it is point-anti-continuous too.
\end{theorem}

\begin{proof}
Assume that $x \in X$, $V \in \mathcal{S}$ and $f(x) \in V$. Now $f^{-1}(V) \in \mathcal{T}$ and $x \in f^{-1}(V)$. Moreover, $f(f^{-1}(V)) \subseteq V$. 
\end{proof}

The converse need not to be true. Let $X = \mathbb{N}$ and $\mathcal{T} = \{ \{0\}, \{1\}, \{2\}, \{3\}, ... \}$. Let $Y = \mathbb{R}^{+} \cup \{0\}$ and $\mathcal{S}$ be a collection of all closed intervals of length $1$ such that their endpoints are natural numbers. Let $f: X \to Y$ be defined as $f(n) = n$. Assume now that $k \in \mathbb{N}$, $V \in \mathcal{S}$ and $f(k) \in V$. Then $f(k)$ must be left or right endpoint of an interval $V$. Then there is $U \in \mathcal{T}$, namely $U = \{k \}$ such that $f(U) = \{k\} \subseteq V$. Clearly, $k \in U$. However, if $V = [a, b]$ for some natural $a, b$, then $f^{-1} = \{a, b\} \notin \mathcal{T}$. 

\section{Density and nowhere density}

Let us interpret the notions of density and nowhere density in anti-topological environment. 

\subsection{Anti-density}

The first definition is standard. It relies on the idea of $X$ as the closure of our set.

\begin{definition}
Let $(X, \mathcal{T})$ be an anti-topological space and $A \subseteq X$. We say that $A$ is \emph{anti-dense} if and only if $aCl(A) = X$. 
\end{definition}

\begin{lemma}
Let $(X, \mathcal{T})$ be an anti-topological space and $A \subseteq X$. $A$ is anti-dense if and only if $\mathfrak{Z} = \{B \in \mathcal{T}_{Cl}; A \subseteq B\} = \emptyset$.
\end{lemma}

\begin{proof}
If $Z$ is empty, then $aCl(A) = \bigcap \mathfrak{Z} = \bigcap \emptyset = X$. Now assume that $aCl(A) = X$ and $\mathfrak{Z} \neq \emptyset$. Thus $\bigcap \mathfrak{Z} = X$. This is possible only if for any $B \in \mathfrak{Z}$, $B = X$. But $X \notin \mathcal{T}_{Cl}$ (because $\emptyset \notin \mathcal{T}$). 
\end{proof}

The next theorem gives us alternative interpretation of anti-density.

\begin{theorem}
Let $(X, \mathcal{T})$ be an anti-topological space and $A \subseteq X$. Then $A$ is anti-dense if and only if $A \cap B \neq \emptyset$ for any $B \in \mathcal{T}$. 
\end{theorem}

\begin{proof}

Let $aCl(A) = X$ and $B \in \mathcal{T}$. Assume that $A \cap B = \emptyset$. Let $x \in B$ (clearly, $B$ is non-empty because it is anti-open). Then $x \in X = aCl(A)$. Thus $A \cap B \neq \emptyset$ (see Lemma \ref{properties} (12)). 

Assume now that $A \cap B \neq \emptyset$ for any $B \in \mathcal{T}$. Suppose that $A$ is not anti-dense. This means that $aCl(A) \neq X$. Hence, there is some $D \in \mathcal{T}_{Cl}$ such that $A \subseteq D$. Of course $D \neq X$. Consider $-D = X \setminus D$. Clearly, $-D \in \mathcal{T}$ and $A \cap (-D) = \emptyset$. Contradiction.

\end{proof}

\begin{remark}
We could use the following reasoning to prove right-to-left direction in the preceding theorem. Assume that $A \cap B \neq \emptyset$ for any $B \in \mathcal{T}$. Suppose that $A$ is not anti-dense. Hence, $aCl(A) \neq X$. Take some $x$ such that $x \in X \setminus aCl(A)$. Now $x \notin aCl(A)$. By the very definition of anti-closure it means that there exists certain anti-open $U$ such that $x \in U$ and $A \cap U = \emptyset$. This is contradiction because $A$ has non-empty intersection with each set from $\mathcal{T}$. 

It seems that Modak used this reasoning in \cite{WEAK} (Theorem 3.1.) with respect to weak structures.
\end{remark}

\begin{lemma}
Let $(X, \mathcal{T})$ be an anti-topological space and $A \subseteq X$. Then $A$ is anti-dense if and only if $aInt(-A) = \emptyset$.

\end{lemma}

\begin{proof}
Let $aInt(-A) = \emptyset$. We know that $aInt(-A) = -aCl(A)$. Hence, $-aCl(A) = \emptyset$ and thus $aCl(A) = X$. Now let $aCl(A) = X$. Then $-aCl(A) = \emptyset$. However, $-aCl(A) = aInt(-A)$. 

Note that we used some properties from Lemma \ref{properties}.

\end{proof}

\begin{lemma}
Let $(X, \mathcal{T})$ be an anti-topological space. Assume that $A, B \subseteq X$ are two anti-dense sets. Then their union is anti-dense too. 
\end{lemma}

\begin{proof}
We know that $X = aCl(A) = aCl(B)$. Of course $A \subseteq (A \cup B)$. Hence $X = aCl(A) \subseteq aCl(A \cup B)$. But then $aCl(A \cup B)$ must be $X$. 
\end{proof}

As for the intersection of two anti-dense setse, it can be anti-dense or not. For example, if $X = \{1, 2, 3, 4\}$ with $\mathcal{T} = \{ \{1, 2\}, \{2, 3\}, \{3, 4\} \}$, then we can consider two anti-dense sets $\{1, 2, 3\}$ and $\{2, 3, 4\}$. Their intersection $\{2, 3\}$ is anti-dense too. On the other hand, consider $X = \{a, b, c, d, e\}$ and $\mathcal{T} = \{ \{a, b\}, \{c, d\}, \{e\} \}$. Take $\{a, c, e\}$ and $\{b, d, e\}$. They are both anti-dense but their intersection is $\{e\}$ and this set is not anti-dense (for example, it does not have non-empty intersection with $\{c, d\}$). 

\subsection{Anti-nowhere density}
Now we may briefly discuss the idea of anti-nowhere density. 

\begin{definition}
Let $(X, \mathcal{T})$ be an anti-topological space and $A \subseteq X$. We say that $A$ is \emph{anti-nowhere-dense} if and only if $aInt(aCl(A)) = \emptyset$. 
\end{definition}

\begin{example}
Let $(X, \mathcal{T})$ be like in Example \ref{exlist} (1). Consider $A = \{1, 4\}$. Now $aCl(A) = \{1, 4\}$ and $aInt(\{1, 4\}) = \bigcup \emptyset = \emptyset$. This set is anti-nowhere dense. However, it does not mean that for any $B \in \mathcal{T}$ we shall find such anti-open $C$ contained in $B$ that $A \cap C = \emptyset$. For example, if $B = \{1, 2\}$, then $A \cap \{1\} \neq \emptyset$ and $\{2\} \notin \mathcal{T}$. 
\end{example}

The last example shows us that it would not be sensible to define anti-nowhere density in terms of empty intersection with at least one anti-open subset of each anti-open set. As we menioned before, anti-open sets do not have proper anti-open subsets.

Of course, in some anti-topologies there are sets which have empty intersection with \emph{any} anti-open set. Consider $X = \{a, b, c, d, e, f\}$ with $\mathcal{T} = \{ \{a, b\}, \{c, d\}, \{e\} \}$. Think about $\{f\}$. Its intersection with any anti-open set is empty. 

Another example: $X = \mathbb{N}^{+}$, $\mathcal{T} = \{ \{1, 3\}, \{5, 7\}, \{9, 11\}, \{13, 15\}, ... \}$. Consider $A = \{2, 4, 6, 8, 10, ...\}$. Its intersection with any anti-open set is empty. Now consider the same universe and $\mathcal{S} = \{ \{1, 2\}, \{2, 3\}, \{3, 4\}, \{4, 5\}, ... \}$. Now it is not possible to find non-empty set with non-empty intersection with every set from $\mathcal{S}$. 

Some other results on (anti-)density, nowhere density and rarity have been presented in \cite{WITCZAK}. 

\section{Conclusion and future work}
In this paper we interpreted some of the basic topological notions in anti-topological setting. In this way we extended the initial research presented in \cite{SAHIN}. We think that now it would be valuable to explore both the topics covered in our paper and those not presented. For example, one can think about separation axioms and compactness. Also the study of semi-, pre-, $\alpha$-, $\beta$, regular and $b$-anti-open sets should be performed. Moreover, it might be interesting to treat anti-topologies as semantic models for some classes of modal logics. In \cite{WIT} we presented such an approach with respect to generalized infra-topological spaces. Here we would like to sketch briefly some general and initial suppositions. 

Perhaps the most natural way would be to consider anti-topological possible-world semantics. Hence we should assume that $W \neq \emptyset$ (namely, the universe of possible worlds denoted by $w, v, u...$) is equipped with some anti-topology $\mathcal{T}$. Now the interpretation\footnote{We assume tacitly that these considerations are intended for those readers who are somewhat familiar with elementary tools of formal propositional logic. Hence, they should think that the interpretation of propositional variables is standard. The same with Boolean operators (including implication). Clearly, we assume that our logic is \emph{classical} in the sense that it satisfies the law of the excluded middle. In particular, it is not intuitionistic (however, we could consider this case).} of modality could be as follows: 

$w \Vdash \Box \varphi$ $\Leftrightarrow$ $V(\varphi) \in \mathcal{T}$, where $V(\varphi) = \{z \in W: z \Vdash \varphi\}$. 

Now let us consider the following formula: $\gamma := \Box \varphi \land \Box \psi \to \lnot \Box (\varphi \lor \psi)$. Assume that $\gamma$ is \emph{not} a tautology. Hence there is a model with some world $w$ such that $w \Vdash \Box \varphi$ and $w \Vdash \Box \psi$. Thus, both $V(\varphi)$ and $V(\psi)$ belong to $\mathcal{T}$. On the other hand, $w \nVdash \lnot \Box (\varphi \lor \psi)$, which means that $w \Vdash \Box (\varphi \lor \psi)$. Hence $V(\varphi \lor \psi) \in \mathcal{T}$. However, $V(\varphi \lor \psi) = V(\varphi) \cup V(\psi)$, so it cannot be anti-open (as a union of two different anti-open sets). Thus, $\gamma$ is a tautology. This result seems to be rather surprising and counter-intuitive. 

As for the notion of compactness (in its topological meaning), then it can be applied to anti-topologies too. However, one should remember that any open cover consisting of more than one set cannot be anti-open. Be as it may, we would like to apply the theory of \emph{abstract spectra} within anti-topological framework (see \cite{DIMOV}). In its original form, this theory shows some interesting applications of the interplay between two topologies on the same space (and their subfamilies). 

Finally, we think that one can look for applications of anti-topologies in the field of data clustering, knowledge bases and formal concept analysis. We think that the idea of exclusion (namely, the idea of anti-closure of our family under some operations, namely unions and intersections) may be studied in this context.

\end{document}